\documentclass[oneside]{amsart}
\usepackage{amsmath,amssymb,verbatim}
\usepackage{amsfonts}
\usepackage{amsthm}
\usepackage{hyperref}
\usepackage{amscd}
\usepackage{color}
\usepackage{enumerate}
\usepackage{xypic}
\usepackage{tikz}
\usepackage{xfrac}
\usepackage{subcaption}
\newcommand{\ddiv}{\text{div}}
\newcommand{\DDIV}{\text{Div}}

\newcommand{\ZZ}{\mathbb{Z}}
\newcommand{\GG}{\mathbb{G}}

\newcommand{\RR}{\mathbb{R}}

\newcommand{\internal}{\mathcal{I}_\Gamma}

\newcommand{\self}{\calD_{\Gamma}}

\newcommand{\calR}{\mathcal{R}}
\newcommand{\calD}{\mathcal{D}}
\newcommand{\calE}{\mathcal{E}}

\DeclareMathOperator{\trop}{Trop}

\newtheorem{maintheorem}{Theorem}

\newtheorem{theorem}{Theorem}[section]
\newtheorem{proposition}[theorem]{Proposition}
\newtheorem{lemma}[theorem]{Lemma}
\newtheorem{corollary}[theorem]{Corollary}

\newtheorem{conjecture}[theorem]{Conjecture}
\newtheorem{definition}[theorem]{Definition}
\newtheorem{question}[theorem]{Question}

\theoremstyle{remark}
\newtheorem{example}[theorem]{Example}
\newtheorem{remark}[theorem]{Remark}

\newtheorem*{acknowledgements}{Acknowledgements}

\newcommand{\matt}[1]{\textcolor{green!60!black}{[[$\spadesuit\spadesuit\spadesuit$ #1]]}}

\title{Lifting tropical self intersections}

\author{{\larger Y}{\smaller oav} {\larger L}{\smaller en}}
\address{Yoav Len\\ Department of Mathematics\\ Georgia Institute of Technology\\ Atlanta, GA 30332\\ USA}
\email{yoav.len@math.gatech.edu}
\author {{\larger M}{\smaller atthew} {\larger S}{\smaller atriano}}
\address{Matthew Satriano\\ Department of Pure Mathematics\\ University of Waterloo\\ Waterloo, ON N2L 3G1\\ Canada}
\email{msatrian@uwaterloo.ca}

\begin{document}
 
\begin{abstract}
We study the tropicalization of  intersections of plane curves, under the assumption that they have the same tropicalization. We show that the  set of tropical divisors that arise in this manner is a pure dimensional balanced polyhedral complex and compute its dimension. When the genus is at most $1$, we show that all the tropical divisors that move in the expected dimension are realizable. 
As part of the proof, we introduce a combinatorial tool for  explicitly constructing large families of realizable tropical divisors.

\end{abstract}
\maketitle
\setcounter{tocdepth}{1}
\tableofcontents

\section{Introduction}
This paper is motivated by the following question: given plane curves $C$ and $C'$, which divisors supported on $\trop(C)\cap\trop(C')$ are tropicalizations of divisors from $C\cap C'$? When the two tropical curves intersect properly, every tropical intersection point is the tropicalization of an algebraic intersection point, counted with multiplicity \cite{OP}. However, as is often the case, tropical curves do not intersect properly, and their intersection may only be lifted at the level of cycles \cite{OR, He}. In this paper we focus on the other extreme: where the two curves have the exact same tropicalization.

Fix  a tropical curve $\Gamma$ in $\RR^2$, and an algebraic curve $C$ satisfying $\trop(C)=\Gamma$. A divisor $\calD$ on $\Gamma$ is said to be $C$-realizable (see Definition \ref{def:realizable}) if it is the tropicalization of the intersection of $C$ with another curve $C'$, such that $\trop(C')=\trop(C)=\Gamma$.
A divisor is said to be \emph{rational} if its coordinates are in the valuation group of the ground field. Clearly,  rationality is a necessary condition for realizability.

Our first main result is the following. 
\begin{maintheorem}\label{thm:polyhedral}
Let $\Gamma=\trop(C)$ be a smooth tropical plane curve.  If the valuation of the ground field surjects onto $\RR$, then the set of $C$-realizable divisors is a polyhedral complex of pure dimension $d-g$, where $d$ is the degree of the self intersection $\Gamma\cdot\Gamma$.
\end{maintheorem}
\noindent We should clarify what we mean by `balanced'. Since our tropical curve is embedded in $\RR^2$, we may identify its divisors of degree $d$ with subsets of $(\RR^2)^d/S^{d}$. We therefore refer to a set as balanced if its pullback to $\RR^{2d}$ is. 

We next specialize to the case where the curve has genus is at most $1$. We prove that a large class of divisors satisfying a combinatorial criterion are necessarily $C$-realizable. 

\begin{definition}\label{def:internal}
Let $\Gamma$ be a tropical curve of genus $1$. 
A divisor in the linear system $|\Gamma\cdot\Gamma|$ is said to be \emph{internal} if has at least two chips (see Section \ref{sec:divisors}) on the cycle, or a single chip in the interior of an edge of the cycle. If $\Gamma$ has genus 0 then every divisor is said to be internal.
\end{definition}

For instance, the divisors described in Figures \ref{subfig:1a} and \ref{subfig:1b} are internal, whereas the divisor in Figure \ref{subfig:1c} is not. 

\begin{figure}[h]
\centering
\begin{tikzpicture}[scale=.5]


\begin{scope}[shift={(0,0)}, scale=1]
\draw[blue] (-1.5,-1) -- (1.5,-1) -- (1.5,1) -- (-1.5,1) -- (-1.5,-1);
\draw[blue] (-1.5,-1) -- (-2.5,-2);
\draw[blue] (1.5,1) -- (2.5,2);
\draw[blue] (1.5,-1) -- (2.5,-2);
\draw[blue] (-1.5,1) -- (-2.5,2);

\draw [fill=purple,radius=.1] (-1.2,1) circle;
\draw [fill=purple,radius=.1] (1.2,1) circle;
\draw [fill=purple,radius=.1] (-.7,-1) circle;
\draw [fill=purple,radius=.1] (.7,-1) circle;

\node at (0,-2.5) {\parbox{0.3\linewidth}{\subcaption{}\label{subfig:1a}}};
\end{scope}

\begin{scope}[shift={(7,0)}, scale=1]
\draw[blue] (-1.5,-1) -- (1.5,-1) -- (1.5,1) -- (-1.5,1) -- (-1.5,-1);
\draw[blue] (-1.5,-1) -- (-2.5,-2);
\draw[blue] (1.5,1) -- (2.5,2);
\draw[blue] (1.5,-1) -- (2.5,-2);
\draw[blue] (-1.5,1) -- (-2.5,2);

\draw [fill=purple,radius=.1] (-1.7,-1.2) circle;
\draw [fill=purple,radius=.1] (.5,1) circle;
\draw [fill=purple,radius=.1] (1.7,1.2) circle;
\draw [fill=purple,radius=.1] (2.1,1.6) circle;

\node at (0,-2.5) {\parbox{0.3\linewidth}{\subcaption{}\label{subfig:1b}}};
\end{scope}

\begin{scope}[shift={(14,0)}, scale=1]
\draw[blue] (-1.5,-1) -- (1.5,-1) -- (1.5,1) -- (-1.5,1) -- (-1.5,-1);
\draw[blue] (-1.5,-1) -- (-2.5,-2);
\draw[blue] (1.5,1) -- (2.5,2);
\draw[blue] (1.5,-1) -- (2.5,-2);
\draw[blue] (-1.5,1) -- (-2.5,2);

\draw [fill=purple,radius=.1] (1.5,1) circle;
\draw [fill=purple,radius=.1] (1.8,-1.3) circle;
\draw [fill=purple,radius=.1] (-1.7,-1.2) circle;
\draw [fill=purple,radius=.1] (-2,1.5) circle;

\node at (0,-2.5) {\parbox{0.3\linewidth}{\subcaption{}\label{subfig:1c}}};
\end{scope}

\end{tikzpicture}
\caption{Three  divisors in $|\self|$.}
\label{ex:internal}
\end{figure}


Our next main result is then:

\begin{maintheorem}\label{thm:main}
Suppose that  $\Gamma=\trop(C)$ is a smooth tropical plane curve of genus $g\leq 1$. Then every rational internal divisor is $C$-realizable.  In particular, every divisor in $|\Gamma\cdot\Gamma|$ that is supported on the minimal skeleton is realizable. 
\end{maintheorem}

An inherent difficulty in the proof is that, \emph{a priori}, the set of curves with a fixed tropicalization is a non-Archimedean semi-algebraic set that does not have a nice algebraic structure \cite{NPS}: it is not Zariski closed, and is not closed under addition. However, using Lemma \ref{lem:anyCurve} below, we can replace it with a  complete linear system without changing the set of intersection points. This may be considered as a variation of tropical modification \cite{CM14, Mi06}.  
The realizability problem is therefore replaced with characterizing the tropicalization of certain linear systems. While the literature contains various advancements towards this goal in the case of abstract curves \cite{CJP, CLM, MUW}, this is in general a wide open problem. In the present paper, however, we gain additional mileage by explicitly using the polyhedral structure obtained from the embedding of $\Gamma$ into $\RR^2$.

We stress, however, that Theorem \ref{thm:main} also ensures the realizability of divisors supported on the minimal skeleton of $\Gamma$, and such divisors \emph{do not} depend on the embedding of $\Gamma$ into $\RR^2$.



In \cite{Mor15}, Morrison provides a strong necessary condition for lifting intersections of smooth tropical planes curves in terms of chip firing: if $\calD$ is the tropicalization of the scheme-theoretic $C\cap C'$, then $\calD$ is linearly equivalent to the stable intersection $\trop(C)\cdot\trop(C')$. By combining the tools developed in Sections \ref{sec:distinguished-loci} and \ref{sec:genus1-case}, we show in Theorem \ref{thm:counter} that the converse of this statement is false.

Finally, we mention that tropical geometry has already had many applications in computing intersections of algebraic varieties: it has been used to compute Gromov--Witten invariants \cite{Rau}, Hurwitz numbers \cite{CMR}, and bitangents of plane curves \cite{LM, IL}. We hope that the results in this paper, and the tools developed to achieve them will play a role in understanding non-transverse tropical intersection, and inspire additional applications of tropical intersection theory to algebraic geometry.

\begin{acknowledgements}
We thank Matt Baker, Dustin Cartwright, Eric Katz,  Sam Payne,  Martin Ulirsch, Joseph Rabinoff, and Josephine Yu for helpful discussions and valuable advice. The second author is partially supported by an NSERC Discovery grant. We also thank the referee for their insightful comments and remarks. 
\end{acknowledgements}


\section{Background}

We begin by setting notations and reminding the reader some of the basic terms in tropical geometry.
By a \emph{tropical curve} we mean a subset $\Gamma\subset\RR^2$ with the following properties. First there is a distinguished finite set of points $v_1,\dots,v_m\in\Gamma\cap\ZZ^2$ which we refer to as \emph{vertices}. For each $i<j$, let $e_{ij}$ be the line segment between $v_i$ and $v_j$. If $e_{ij}\subset\Gamma$, we refer to $e_{ij}$ as a \emph{bounded edge} of $\Gamma$. If a ray $r$ with endpoint at vertex $v_i$ is contained in $\Gamma$, we refer to $r$ as an \emph{unbounded edge} of $\Gamma$; we require that every  edge has rational slope. Then $\Gamma$ is the union of its vertices, bounded edges, and unbounded edges. We further require that at each vertex $v\in\Gamma$, the \emph{balancing condition} holds: $\sum_{v\in e}u_e=0$, where the sum runs through the edges $e$ containing $v$, and $v+u_e$ is the primitive lattice point along the ray emanating from $v$ in the direction of $e$. 

The \emph{dual of a vertex} $v$ is a polygon $P_v$ whose edges have lattice length $1$ and are perpendicular to the edges emanating from the vertex. Such a polygon exists by the balancing condition. The union $P_\Gamma:=\bigcup_{v\in\Gamma} P_v$ is again a polygon \cite[Section 2.2]{BIMS}. We refer to $P_\Gamma$ as the \emph{dual polygon} of $\Gamma$, and refer to the collection $\{P_v\}$ as the \emph{subdivided dual polygon} of $\Gamma$. 
A tropical curve is \emph{smooth} if it is trivalent, and each $P_v$ is a triangle of area $\frac{1}{2}$. Unless said otherwise, we will always assume that tropical curves are smooth. This is also the reason that we do not discuss weights on the edges of the curve.


Let $K$ be an algebraically closed field of characteristic 0 with a non-trivial non-Archimedean valuation $\nu:K\to\RR$. Unless stated otherwise, we assume that the valuation is surjective onto $\RR$. 
The \emph{tropicalization} of a variety $X$ embedded in an $n$-dimensional torus $\mathbb{G}_m^n$ is the set of points
\[
\trop(x_1,\ldots,x_n)=(-\nu(x_1),\ldots,-\nu(x_n)),
\]
as $(x_1,\ldots,x_n)$ runs through the closed points of $X$. Note that if we removed the assumption that the valuation surjected onto $\RR$, the tropicalization would be the closure of this map.
When $n=2$ and $X$ is a curve, Kapranov's theorem \cite[Theorem 3.13]{MS} implies that $\trop(X)$ is a tropical curve. Moreover, every tropical curve in $\RR^2$ arises this way. 

When two tropical curves $\Gamma$ and $\Gamma'$ intersect properly at a point $p$, their intersection multiplicity is $|\!\det(u,u')|$, where $u,u'$ are the primitive direction vectors  of the edges containing $p$. The \emph{stable intersection} of $\Gamma$ and $\Gamma'$, denoted $\Gamma\cdot\Gamma'$, is obtained by choosing a vector $v$ such that $\Gamma$ and $\epsilon v+\Gamma'$ intersect properly for small enough $\epsilon>0$, and taking the limit of $\Gamma\cap(\epsilon v + \Gamma')$ as $\epsilon$ tends to $0$ \cite[Section 3.6]{MS}.

\subsection{Tropical and algebraic divisors}\label{sec:divisors}
We provide a brief review of the basic concepts in the theory of tropical divisors. See \cite{BJ} for a more thorough treatment. 
A \emph{divisor} on $\Gamma$ is a formal sum 
\[
\calD = a_1 p_1 + \ldots +a_k p_k,
\]
where $p_1,p_2,\ldots,p_k$ are points of $\Gamma$ and $a_1,a_2,\ldots,a_k$ are integers. In this case, $\calD$ is said to have $a_i$ \emph{chips} at $p_i$; the \emph{support} of $\calD$ is the set of points $p_i$. The divisor $\calD$ is said to be \emph{effective} if all $a_i\geq0$. The \emph{degree} of $\calD$ is defined as $\deg(\calD)=\sum a_i$. The divisor group of $\Gamma$, denoted $\DDIV(\Gamma)$, is the free abelian group generated by the points of $\Gamma$. 

Suppose $\phi\colon\Gamma\to\RR$ is a continuous function whose restriction to each edge of $\Gamma$ is a piecewise linear function with integer slopes and finitely many domains of linearity. Assume moreover that $\phi$ is constant outside an arbitrary large disk.
We associate to $\phi$ a divisor $\ddiv(\phi)$ as follows. If $p$ is in the interior of a domain of linearity of $\phi$, then it is not in the support of $\ddiv(\phi)$. Otherwise, let $\{\ell_i\}$ be the domains of linearity with $p\in\ell_i$, and consider each $\ell_i$ as an edge oriented toward $p$; then the number of chips of $\ddiv(\phi)$ at $p$ is defined to be the sum of slopes of the $\phi|_{\ell_i}$. Divisors of the form $\ddiv(\phi)$ are referred to as \emph{principal}. We say that two divisors are \emph{linearly equivalent} if their difference is principal.

The \emph{linear system} of a divisor $\calD$, denoted $|\calD|$ is the set of effective divisors that are linearly equivalent to $\calD$. It is a polyhedral complex, which in general does not have pure dimension \cite{HMY}. We may view $|\calD|$ as a subset of $\Gamma^d/{S_d}$, where $S_d$ is the group of permutations on $d$ elements. In particular, it is a subset of $\RR^{2d}/{S^d}$. For any subset $\mathcal{S}$ of $\Gamma^d/S_d$, we denote $\widetilde{\mathcal{S}}$ its pullback to $\Gamma^d$. 


Fix a divisor $D$ of degree $d$ on a smooth curve $C$ embedded in the 2-dimensional torus $\mathbb{G}_m^2$. Since $C$ is not a proper curve, we need to be careful with regards to the definition of the linear system. 
Given a smooth compactification $\overline{C}$ of $C$, the linear system $|D|$ consists of the effective divisors that are are equivalent to $D$ on $\overline{C}$, and are supported on $C$.

%
%
%



%
%

\section{Realizing divisors via stable intersection}
\label{sec:distinguished-loci}
For the rest of this section, we fix a curve $C\subseteq \GG_m^2$ defined over a non-Archimedean valued field $(K,\nu)$ with smooth tropicalization $\Gamma\subseteq\RR^2$. We denote the genus of $\Gamma$ by $g=g(\Gamma)$, the stable self intersection $\Gamma\cdot\Gamma$ by $\self$, and the degree of $\self$ by $d$. 

\begin{definition}
\label{def:realizable}
A divisor $\calD$ in $|\self|$  is $C$-\emph{realizable} (or just realizable when there is no cause for confusion) if there is a curve $C'$ such that $\trop(C')=\Gamma$ and $\calD=\trop(C\cap C')$. 
\end{definition}


We denote by $\calR_C \subset |\self| \subset \Gamma^d/S_d$ the set of all $C$-realizable divisors. Let $\widetilde\calR_C\subset\Gamma^d$ be the pullback of $\calR_C$ under the quotient map $\Gamma^d\to\Gamma^d/S_d$. 
Our goal in this section is to show  that $\widetilde\calR_C$ is a balanced polyhedral complex of dimension $d-g$, i.e.~to prove Theorem \ref{thm:polyhedral}. We accomplish this in several steps. In Section \ref{subsec:d-g-dimensional-locus}, we study the auxiliary locus $\calR_\Gamma^{st}$ of divisors $\Gamma\cdot\Gamma'$, where $\Gamma'$ varies through all the tropical curves whose subdivided dual polygon coincides with the subdivided dual polygon of $\Gamma$. 
We prove in Proposition \ref{prop:degreesOfFreedom} that this yields a $(d-g)$-dimensional sub-locus of $\widetilde\calR_C$. This results in Theorem 
\ref{thm:polyhedral} as well as the genus $0$ case of Theorem \ref{thm:main}.

%
%
%
%
%
%
%

\subsection{Obtaining a $(d-g)$-dimensional locus of realizable divisors}
\label{subsec:d-g-dimensional-locus}

Although we are interested in studying  divisors of the form $\trop(C\cap C')$ where $\trop(C')=\Gamma$, the following lemma shows that we are allowed to consider a larger, better behaved set.

\begin{lemma}\label{lem:anyCurve}
Let $C_1$ and $C_2$ be distinct  curves in $\mathbb{G}_m^2$, such that the Newton polygon of $C_2$ is contained in the Newton polygon of $C_1$. Then there is a curve $C'_2$ with $\trop(C_1)=\trop(C'_2)$ such that the scheme-theoretic intersections $C_1\cap C_2$ and  $C_1\cap C'_2$ are equal.
\end{lemma}
\begin{proof}
Let $f_1,f_2\in K[x^\pm,y^\pm]$  such that $C_1=V(f_1)$ and $C_2=V(f_2)$. Let $C'_2=V(h)$ where $h=f_1 + t^r f_2$, $t$ is a uniformizer of $K$, and $r$ is a positive integer. Since the Newton polygon of $C_2$ is contained in that of $C_1$, choosing $r$ sufficiently large, we have $\trop(C_1)=\trop(C'_2)$. Since we have an equality of ideals $(f_1,f_2)=(f_1,h)$, the scheme-theoretic intersections $C_1\cap C_2$ and $C_1\cap C'_2$ are equal.
\end{proof}


In light of Lemma \ref{lem:anyCurve}, we make the following definitions, where ``st'' stands for ``stable''.

\begin{definition}
\label{def:same-subdivided-dual-polygon}
 $\calR_{\Gamma}^{st}$ is the set of divisors $\Gamma\cdot\Gamma'$, where $\Gamma'$ is any tropical curve whose dual polygon coincides with that of $\Gamma$. 
\end{definition}

Lemma \ref{lem:anyCurve} immediately implies:

\begin{corollary}
\label{cor:anyCurve}
We have a containment
\[
\calR_{\Gamma}^{st}\subseteq\calR_C.
\]
\end{corollary}
\begin{proof}
Let $\calD=\Gamma\cdot\Gamma'$ where $\Gamma'$ has the same dual polygon as $\Gamma$. Choose $C'$ such that $\trop(C')=\Gamma'$. Lemma \ref{lem:anyCurve} then tells us that there is a curve $C''$ such that $C\cap C'=C\cap C''$ and $\trop(C'')=\trop(C)=\Gamma$. In particular, $\calD=\trop(C\cap C'')\in\calR_C$.
\end{proof}

As the next example shows, $\calR_{\Gamma}^{st}$ does not contain all the realizable divisors. Nonetheless, Proposition \ref{prop:degreesOfFreedom} will show that it contains a large dimensional set of divisors.

\begin{example}[{$\calR_{\Gamma}^{st}\neq\calR_C$}]
\label{ex:stable-int-realizes-large-diml-set}
Let $t$ be a uniformizer for $K$ and let $C$ be the zero set of $f:=t+x+y+txy$. Its tropicalization $\Gamma$, depicted in Figure \ref{ex:tropicalCurve}, has vertices at $(-1,-1)$ and $(1,1)$.

Let $\calD = p_1 + p_2$, where $p_1$ and $p_2$ is any pair of points on the diagonal edge of $\Gamma$. Then we may find a tropical curve with the same dual polygon, whose stable intersection with $\Gamma$ consists of $p_1$ and $p_2$, see Figure \ref{subfig:2a}. By Corollary \ref{cor:anyCurve}, $\calD$ is $C$-realizable. Similarly, Figure \ref{subfig:2b} illustrates how Corollary \ref{cor:anyCurve} realizes divisors that consist of a chip on the vertical ray emanating from $(-1,-1)$ and a chip on the horizontal ray emanating from $(1,1)$.

In Figure \ref{subfig:2c} we use the same technique to realize 2 chips at a single point along the diagonal edge of $\Gamma$. Notice that the stable intersection of the two tropical curves has 2 chips at their point of intersection since the tangent directions there generate a sublattice of index 2. Observe that, although their \emph{subdivided} dual polygons differ, the dual polygons of the two tropical curves are the same; hence Corollary \ref{cor:anyCurve} still applies here.

Finally, there are many divisors in $\calR_C$ that are not in $\calR_{\Gamma}^{st}$, such as the one depicted in Figure \ref{subfig:2d}. There is no curve with the same dual polygon as $f$ whose stable intersection with $\Gamma$ is $(1,2) + (1,3)$. On the other hand, this divisor is in $\calR_C$ since it is the tropicalization of the intersection of $C$ with the curve 
\[
g = ta + bx + cy + tdxy,
\]
where 
\[
b=a+\beta-t\gamma-t^2\beta-t^5\delta,\quad c=a+\beta-t\gamma-t^5\delta,\quad d=a+\beta-t\gamma
\]
and $a,\beta,\gamma,\delta$ are generic with valuation 0.

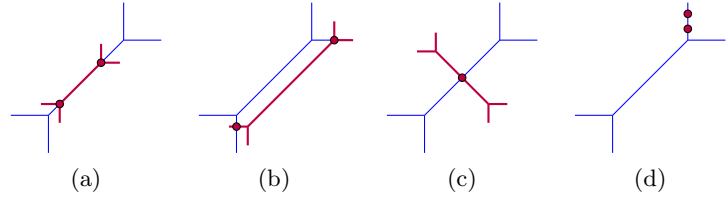
\begin{figure}[h]
\centering
\begin{tikzpicture}[scale=.5]


\begin{scope}[shift={(5,0)}]
\draw[blue] (-1,-1) to (1,1);
\draw[blue] (-2,-1) to (-1,-1);
\draw[blue] (-1,-2) to (-1,-1);
\draw[blue] (1,1) to (2,1);
\draw[blue] (1,1) to (1,2);

\draw[purple, thick] (-0.7,-0.7) to (0.4,0.4);
\draw[purple, thick] (-0.7,-0.7) to (-1.2,-0.7);
\draw[purple, thick] (-0.7,-0.7) to (-0.7,-1.2);
\draw[purple, thick] (0.4,0.4) to (0.9,0.4);
\draw[purple, thick] (0.4,0.4) to (0.4,0.9);

\draw [fill=purple,radius=.1] (-0.7,-0.7) circle;
\draw [fill=purple,radius=.1] (0.4,0.4) circle;
\node at (0,-2.5) {\parbox{0.3\linewidth}{\subcaption{}\label{subfig:2a}}};
\end{scope}

\begin{scope}[shift={(10,0)}]
\draw[blue] (-1,-1) to (1,1);
\draw[blue] (-2,-1) to (-1,-1);
\draw[blue] (-1,-2) to (-1,-1);
\draw[blue] (1,1) to (2,1);
\draw[blue] (1,1) to (1,2);

\draw[purple, thick] (-0.7,-1.3) to (1.6,1);
\draw[purple, thick] (-0.7,-1.3) to (-1.2,-1.3);
\draw[purple, thick] (-0.7,-1.3) to (-0.7,-1.8);
\draw[purple, thick] (1.6,1) to (2.1,1);
\draw[purple, thick] (1.6,1) to (1.6,1.5);

\draw [fill=purple,radius=.1] (-1,-1.3) circle;
\draw [fill=purple,radius=.1] (1.6,1) circle;
\node at (0,-2.5) {\parbox{0.3\linewidth}{\subcaption{}\label{subfig:2b}}};
\end{scope}

\begin{scope}[shift={(15,0)}]
\draw[blue] (-1,-1) to (1,1);
\draw[blue] (-2,-1) to (-1,-1);
\draw[blue] (-1,-2) to (-1,-1);
\draw[blue] (1,1) to (2,1);
\draw[blue] (1,1) to (1,2);

\draw[purple, thick] (-0.7,0.7) to (0.7,-0.7);
\draw[purple, thick] (-0.7,0.7) to (-1.2,0.7);
\draw[purple, thick] (-0.7,0.7) to (-0.7,1.2);
\draw[purple, thick] (0.7,-0.7) to (1.2,-0.7);
\draw[purple, thick] (0.7,-0.7) to (0.7,-1.2);

\draw [fill=purple,radius=.1] (0,0) circle;
\node at (0,-2.5) {\parbox{0.3\linewidth}{\subcaption{}\label{subfig:2c}}};
\end{scope}

\begin{scope}[shift={(20,0)}]
\draw[blue] (-1,-1) to (1,1);
\draw[blue] (-2,-1) to (-1,-1);
\draw[blue] (-1,-2) to (-1,-1);
\draw[blue] (1,1) to (2,1);
\draw[blue] (1,1) to (1,2);

\draw [fill=purple,radius=.1] (1,1.3) circle;
\draw [fill=purple,radius=.1] (1,1.7) circle;
\node at (0,-2.5) {\parbox{0.3\linewidth}{\subcaption{}\label{subfig:2d}}};
\end{scope}
\end{tikzpicture}
\caption{Four realizable divisors on $\trop(f)$.}
\label{ex:tropicalCurve}
\end{figure}
\end{example}

Although $\calR_{\Gamma}^{st}\neq\calR_C$, as Example \ref{ex:stable-int-realizes-large-diml-set} illustrates, we now show that it nonetheless contains a large dimensional set of divisors. Our next main goal for the rest of this subsection is to prove:

\begin{proposition} \label{prop:degreesOfFreedom}
Let $\Gamma$ be a tropical plane curve of genus $g$ with $d$ vertices. 
Then $\calR_{\Gamma}^{st}$ contains a $(d-g)$-dimensional set of divisors.
\end{proposition}

In fact, the $(d-g)$-dimensional set of divisors we obtained will be realized by intersecting $\Gamma$ with tropical curves $\Gamma'$ having the same \emph{subdivided} dual polygon as $\Gamma$. Before proceeding with the proof, we illustrate the main ideas with an example.

\begin{example}[{Illustrating the proof of Proposition \ref{prop:degreesOfFreedom}}]
\label{ex:illustrating-prop:degreesOfFreedom}
Let $\Gamma$ be the tropical curve shown in Figure \ref{subfig:3a}. Its vertices are labeled $v_1,\dots,v_4$ and bounded edges are labeled $e_1,\dots,e_4$. Let $\mathcal{S}_\Gamma$ be the set of tropical curves with the same subdivided Newton polygon as $\Gamma$. If $\Gamma'\in\mathcal{S}_\Gamma$, then its vertices $v'_1,\dots,v'_4$ and bounded edges $e'_1,\dots,e'_4$ are in canonical bijection with those of $\Gamma$. Moreover, $\Gamma'$ is completely determined once we fix the location of $v'_4$ and the lengths of the $e'_i$. In Figure \ref{subfig:3b}, we take $\Gamma'$ to be a small translate of $\Gamma$, i.e.~$\Gamma'=\Gamma+\eta$ where $\eta$ is a small generic vector; the blue curve represents $\Gamma$ and the red curve $\Gamma'$.

Now, here is the key point: consider all $\Gamma''\in\mathcal{S}_\Gamma$ that are obtained from $\Gamma'$ by small perturbations to the edge lengths $e'_1,\dots,e'_4$ while keeping the location of $v'_4$ fixed. There is a $2$-dimensional family of such $\Gamma''$, since there are two linear conditions coming from the fact that the cycle must close (a linear condition from the $x$-coordinates and a linear condition from the $y$-coordinates). Corollary \ref{cor:anyCurve} tells us that each such $\Gamma\cdot\Gamma''$ is realizable. We are interested in computing the dimension of the locus of divisors $\Gamma\cdot\Gamma''$ obtained in this manner. In this particular example, we see there is a $1$-dimensional family of such $\Gamma\cdot\Gamma''$: the purple chip pictured in Figure \ref{subfig:3b} is the only chip which can vary in this family. Combined with the fact that we have a $2$-dimensional space of choices for $\eta$, we obtain a $3=d-g$ dimensional locus of  realizable divisors.

In this particular example, we are able to directly see there is a $1$-dimensional family of divisors $\Gamma\cdot\Gamma''$ where $\Gamma''$ is a small perturbation of $\Gamma'$ without moving $v_4'$. More generally, the proof of Proposition \ref{prop:degreesOfFreedom} proceeds as follows. Rather than directly computing this dimension, we instead ask which $\Gamma''$ have the property that $\Gamma\cdot\Gamma''=\Gamma\cdot\Gamma'$. In the example, the curves $\Gamma''$ with this property are precisely those that are obtained from $\Gamma'$ by varying the lengths of $e'_1$ and $e'_3$ while keeping the lengths of $e'_2$ and $e'_4$ constant.  For such $\Gamma''$, the locations of $v'_1$ and $v'_4$ are fixed; in accordance with Definition \ref{def:pinned} below, we say that these vertices are \emph{pinned} since their location is constant throughout the family. We see that the only non-pinned vertices are $v'_2$ and $v'_3$, and these vertices are precisely those which live in the topmost unbounded region of $\RR^2\setminus\Gamma$. This region, which we denote by $\Omega_\eta$, is special since it corresponds to the unique maximal cone of the recession fan which contains $\eta$.

In general, Proposition \ref{prop:degreesOfFreedom} is proven by showing that (i) if $v'_i\notin\Omega_\eta$ then it is pinned, (ii) the $v'_i\in\Omega_\eta$ are not pinned but the location of one such $v'_i\in\Omega_\eta$ determines the location of all other $v'_j\in\Omega_\eta$.  
Hence, the $\Gamma''$ with $\Gamma\cdot\Gamma''=\Gamma\cdot\Gamma'$ vary in a $1$-dimensional family. 
Applying this observation back to the example at hand, the dimension of the space of divisors $\Gamma\cdot\Gamma''$ obtained from perturbations of the edge lengths of $\Gamma'$ is $2-1=1$, which is the same answer we obtained above. Combined with the fact that there are $2$-dimensions worth of choices for $\eta$, we obtained a $1+2=3$ dimensional locus of realizable divisors.

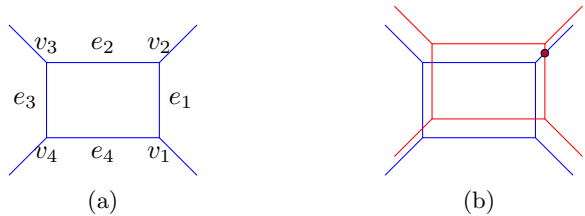
\begin{figure}[h]
\centering
\begin{tikzpicture}[scale=.5]


\begin{scope}[shift={(-5,0)}, scale=1]
\draw[blue] (-1.5,-1) -- (1.5,-1) -- (1.5,1) -- (-1.5,1) -- (-1.5,-1);
\draw[blue] (-1.5,-1) -- (-2.5,-2);
\draw[blue] (1.5,1) -- (2.5,2);
\draw[blue] (1.5,-1) -- (2.5,-2);
\draw[blue] (-1.5,1) -- (-2.5,2);

\node [below] at (1.5,-1) {$v_1$};
\node [above] at (1.5,1) {$v_2$};
\node [above] at (-1.5,1) {$v_3$};
\node [below] at (-1.5,-1) {$v_4$};

\node [right] at (1.5,0) {$e_1$};
\node [above] at (0,1) {$e_2$};
\node [left] at (-1.5,0) {$e_3$};
\node [below] at (0,-1) {$e_4$};


\node at (0,-2.5) {\parbox{0.3\linewidth}{\subcaption{}\label{subfig:3a}}};
\end{scope}

\begin{scope}[shift={(5,0)}, scale=1]
\draw[blue] (-1.5,-1) -- (1.5,-1) -- (1.5,1) -- (-1.5,1) -- (-1.5,-1);
\draw[blue] (-1.5,-1) -- (-2.5,-2);
\draw[blue] (1.5,1) -- (2.5,2);
\draw[blue] (1.5,-1) -- (2.5,-2);
\draw[blue] (-1.5,1) -- (-2.5,2);

\draw[red] (-1.25,-0.5) -- (1.75,-0.5) -- (1.75,1.5) -- (-1.25,1.5) -- (-1.25,-0.5);
\draw[red] (-1.25,-0.5) -- (-2.25,-1.5);
\draw[red] (1.75,1.5) -- (2.75,2.5);
\draw[red] (1.75,-0.5) -- (2.75,-1.5);
\draw[red] (-1.25,1.5) -- (-2.25,2.5);

\draw [fill=purple,radius=.1] (1.75,1.25) circle;

\node at (0,-2.5) {\parbox{0.3\linewidth}{\subcaption{}\label{subfig:3b}}};
\end{scope}

\end{tikzpicture}
\caption{Illustrating the proof of Proposition \ref{prop:degreesOfFreedom}.}
\label{ex:illustrating-degs-of-freedom}
\end{figure}

\end{example}


Having now outlined (in Example \ref{ex:illustrating-prop:degreesOfFreedom}) the general strategy we use to prove Proposition \ref{prop:degreesOfFreedom}, let us begin by establishing some notation. Denote by $E_0(\Gamma)$ the set of bounded edges of $\Gamma$, and let $\mathcal{S}_\Gamma$ be the set of tropical curves $\Gamma'$ with the same subdivided dual polygon as $\Gamma$. Notice that if $\Gamma'\in\mathcal{S}_\Gamma$, then it is obtained by translating $\Gamma$ by some $\eta\in\RR^2$ and then varying the edge lengths of $\Gamma$. Our goal is to understand the space of divisors $\Gamma\cdot\Gamma'$ obtained by letting $\Gamma'$ vary through the elements of $\mathcal{S}_\Gamma$.

We will study $\mathcal{S}_\Gamma$ by breaking it up into more manageable pieces. For each $\Gamma'\in\mathcal{S}_\Gamma$, there is a canonical bijection between the vertices of $\Gamma$ and those of $\Gamma'$, which we denote $\varphi_{\Gamma'}$. For a fixed generic $\eta\in\RR^2$, let $v_\eta\in\Gamma$ be the unique vertex whose dot product with $\eta$ is minimal.\footnote{$v_\eta$ plays the same role as $v_4$ in Example \ref{ex:illustrating-prop:degreesOfFreedom}. Similarly, $v_{\Gamma',\eta}$ plays the role of $v'_4$.}.  Denote $v_{\Gamma',\eta} = \varphi_{\Gamma'}(v_{\eta})$  the vertex of $\Gamma'$ corresponding to $v_\eta$, and let
\[
\mathcal{S}_{\Gamma,\eta}:=\{\Gamma'\in\mathcal{S}_\Gamma\mid v_{\Gamma',\eta}=v_\eta+\eta\}
\]
the set of tropical curves obtained from $\Gamma+\eta$ by varying the edge lengths of $E_0(\Gamma)$ while fixing the position of the vertex $v_\eta+\eta$. That is, we are assigning a length to each element of $E_0(\Gamma)$ subject to the condition that each cycle closes. We may therefore identify $\mathcal{S}_{\Gamma,\eta}$ with an open subset of a vector subspace $V_{\Gamma,\eta}\subset\RR^{|E_0(\Gamma)|}$ of dimension $|E_0(\Gamma)|-2g$. 

Finally, we define  a  piecewise linear map
\[
\Psi_{\Gamma,\eta}:V_{\Gamma,\eta}\to\DDIV(\Gamma)
\]
sending a tropical curve $\Gamma'$ to $\Gamma\cdot\Gamma'$. 

\begin{definition}
\label{def:pinned}
Let $\Gamma'\in\mathcal{S}_{\Gamma,\eta}$, and denote $F=\Psi_{\Gamma,\eta}^{-1} (\Psi_{\Gamma,\eta}(\Gamma'))$, the collection of curves $\Gamma''\in V_{\Gamma,\eta}$ for which $\Gamma\cdot\Gamma'=\Gamma\cdot\Gamma''$. 
 We say that a vertex $v'$ of $\Gamma'$ is $\eta$-\emph{pinned} (or \emph{pinned} when $\eta$ is understood) if there is an open neighborhood $U\subset V_{\Gamma,\eta}\cap F$ of $\Gamma'$ such that $U$ is contained in a domain of linearity of $\Psi_{\Gamma,\eta}$, and $\varphi_{\Gamma''}(v')=v'$ for all $\Gamma''\in U$.
\end{definition}

In other words, a vertex is pinned if moving it changes the stable intersection with $\Gamma$.

\begin{remark}
By the very definition of $\mathcal{S}_{\Gamma,\eta}$, if $\Gamma'$ is in a domain of linearity of $\Psi_{\Gamma,\eta}$, then $v_{\Gamma',\eta}$ is pinned.
\end{remark}

Throughout the rest of this section, we say that two vertices $v'_1$ and $v'_2$ of $\Gamma'$ are \emph{neighbors} if they are connected by an edge of $\Gamma'$.

\begin{lemma}
\label{l:most-verticies-pinned}
Let $\eta\in\RR^2$ be generic with sufficiently small norm, and let $\Gamma'=\Gamma+\eta$. Then the following hold:
\begin{enumerate}
\item $\Gamma'$ is in the interior a domain of linearity of $\Psi_{\Gamma,\eta}$,
\item there is a unique connected component $\Omega_\eta$ of $\RR^2\setminus\Gamma$ such that $\Omega_\eta+\eta\subset\Omega_\eta$,
\item if $v'\in\Gamma'$ is a vertex with $v'\notin\Omega_\eta$, then $v'$ is pinned.
\end{enumerate}
\end{lemma}
\begin{proof}
By our hypothesis on $\eta$, we see $\Gamma'$ intersects $\Gamma$ in $d$ distinct points which lie in the interior of the edges of $\Gamma$. It follows that $\Gamma'$ is in the interior of a domain of linearity.

Consider a connected component $\Omega$ of $\RR^2\setminus\Gamma$. It is convex \cite{SN}, and since $\eta$ is sufficiently small, $(\Omega+\eta)\cap\Omega\neq\varnothing$. If $\Omega$ is a bounded region, then $\Omega+\eta$ is not contained in $\Omega$. By \cite[Corollary 3.10]{recession-cones-form-fan}, the recession cones of $\Gamma$ fit together to form a complete fan, hence $\eta$ is in the interior of a unique maximal cone of the recession fan. As a result, $\Omega+\eta\subset\Omega$ for exactly one unbounded region $\Omega$, proving our second claim.

For the final claim of the lemma, we begin with some trivial observations. Since $\Gamma$ is tropically smooth, it is a balanced trivalent graph. It follows that a vertex $v'\in\Gamma'$ is pinned under any of the following circumstances: (i) $v'$ has two edges that intersect $\Gamma$, (ii) $v'$ has two neighbors that are pinned, (iii) $v'$ has an edge $e'$ that intersects $\Gamma$ as well as a neighbor not coming from $e'$ which is pinned. Finally, we observe that every vertex $v'$ of $\Gamma'$ has at least one edge $e'$ that intersects $\Gamma$. Indeed, $v'=v+\eta$ for some vertex $v$ of $\Gamma$; we choose $e'$ to be the $\eta$-translate of an edge $e$ of $v$ such that the dot product $u_e\cdot\eta<0$, where $u_e$ is the unit vector emanating from $v$ in the direction of $e$.

Now, let $v'_0$ be a vertex of $\Gamma'$ contained in a connected component $\Omega$ of $\RR^2\setminus\Gamma$; assume $\Omega\neq\Omega_\eta$. 
We know $v'_0$ has an edge $e'_0$ intersecting $\Gamma$. If $v'_0$ has two edges intersecting $\Gamma$, then $v'_0$ is pinned. Otherwise, since $\Gamma'$ is trivalent, the remaining two edges of $v'_0$ are contained in $\Omega$. Both of these edges may be unbounded; for any bounded edge, the corresponding neighbor vertex is contained in $\Omega$ and we may apply the same argument to said vertex. The result is a chain of neighboring vertices $v'_{-s},\dots,v'_{-1},v'_0,v'_1,\dots,v'_r$ with the following properties:
\begin{enumerate}
\item Each $v'_i\in\Omega$.
\item Each $v'_i$ has an edge $e'_i$ intersecting $\Gamma$.
\item Let $f'_i$ be the edge connecting $v'_i$ to $v'_{i+1}$. Let $g'_r\neq e'_r,f'_{r-1}$ be the remaining edge of $v'_r$. Let $g'_{-s}\neq e'_{-s},f'_{-s}$ be the remaining edge of $v'_{-s}$. Then $g'_r$ intersects $\Gamma$ or it is unbounded and contained in $\Omega$. Similarly, $g'_{-s}$ intersects $\Gamma$ or it is unbounded and contained in $\Omega$.
\end{enumerate}

We claim that it is impossible for both $g'_r$ and $g'_{-s}$ to be unbounded rays contained in $\Omega$. If $\Omega$ is a bounded region, this is clear. If $\Omega$ is unbounded, then since $v'_{-s}$ and $v'_r$ are $\eta$-translates of vertices of $\Gamma$, the recession cone corresponding to $\Omega$ is generated by the unbounded rays of $g'_{-s}$ and $g'_r$. Since $\Omega\neq\Omega_\eta$, we see $\eta$ is not contained in the recession cone corresponding to $\Omega$, thereby showing that one of these unbounded rays must intersect $\Gamma$. We have therefore established that one of $v'_{-s}$ and $v'_r$ has two edges that intersect $\Gamma$. Without loss of generality, say it is $v'_r$. So, $v'_r$ is pinned. Then $v'_{r-1}$ has a neighbor which is pinned as well as an edge $e'_{r-1}$ which intersects $\Gamma$, so $v'_{r-1}$ is also pinned. Arguing in this manner, we see $v'_r,v'_{r-1},\dots,v'_0$ are all pinned. So, $v'_0$ is pinned, establishing the last of our claims.
\end{proof}

We next analyze the fibers of $\Psi_{\Gamma,\eta}$. This amounts to an analysis of the vertices of $\Gamma'$ contained in $\Omega_\eta$.

\begin{lemma}
\label{l:1-diml-ker-at-Gamma'}
Assume $d>1$, let $\eta\in\RR^2$ be generic with sufficiently small norm, and let $\Gamma'=\Gamma+\eta$. Then there is an open neighborhood $U\subset V_{\Gamma,\eta}$ of $\Gamma'$ such that $U$ is contained in a domain of linearity of $\Psi_{\Gamma,\eta}$ and the fibers of $\Psi_{\Gamma,\eta}|_U$ are $1$-dimensional.
\end{lemma}
\begin{proof}
Let $v_{-s},\dots,v_{-1},v_0,v_1,\dots,v_r$ be the vertices on the boundary of $\Omega_\eta$, and let $v'_i=v_i+\eta\in\Omega_\eta$. Notice that a vertex $v'$ of $\Gamma'$ is in $\Omega_\eta$ if and only if $v'=v'_i$ for some $i$. Since $d>1$, $v_\eta\neq v_i$ for any $i$, and so none of the $v_i$ are \emph{a priori} pinned. As in the proof of Lemma \ref{l:most-verticies-pinned}, there exists an edge $e'_i$ of $v'_i$ that intersects $\Gamma$. 
Let $U$ be a sufficiently small neighborhood of $\Gamma'$ and suppose $\Gamma''\in U$ satisfies $\Gamma\cdot\Gamma''=\Gamma\cdot\Gamma'$, i.e.~$\Gamma'$ and $\Gamma''$ are in the same fiber of $\Psi_{\Gamma,\eta}$. Then Lemma \ref{l:most-verticies-pinned} tells us that $\varphi_{\Gamma''}(v')=v'$ for all vertices $v'\in\Gamma'$ with $v'\neq v'_i$. To prove the fibers of $\Psi_{\Gamma,\eta}|_U$ are $1$-dimensional, we show that once $\varphi_{\Gamma''}(v'_0)$ is chosen, it determines all $\varphi_{\Gamma''}(v'_i)$ and hence determines $\Gamma''$.

Let $p_i$ be the intersection point of $e'_i$ and $\Gamma$, and let $u_i$ denote the unit vector in the direction from $p_i$ to $v'_i$. Let $w_i$ be a unit vector in the direction from $v_i$ to $v_{i+1}$. Since $\Gamma\cdot\Gamma''=\Gamma\cdot\Gamma'$, we know $\varphi_{\Gamma''}(v'_0)=v'_0+\epsilon u_0$ with $|\epsilon|>0$ sufficiently small. Since $u_1$ and $w_1$ are linearly independent, there exists a unique solution to the equation $\varphi_{\Gamma''}(v'_0)+\lambda_1w_1=p_1+\mu_1u_1$ with $\lambda_1,\mu_1\in\RR$; this solution is $\varphi_{\Gamma''}(v'_1)$. Proceeding in this manner we see the $\varphi_{\Gamma''}(v'_i)$ are determined.
\end{proof}

We are now ready to prove the main result of this subsection.

\begin{proof}[{Proof of Proposition \ref{prop:degreesOfFreedom}}]
If $d=1$, the statement is clear, so we may assume $d>1$. Choose $\eta\in\RR^2$ generic with sufficiently small norm, and let $\Gamma'=\Gamma+\eta$. Since $\Gamma'$ lies in the locus of linearity of $\Psi_{\Gamma,\eta}$ by Lemma \ref{l:most-verticies-pinned}, we can apply the First Isomorphism Theorem to compute the image of $\Psi_{\Gamma,\eta}$ in a small neighborhood $U$ of $\Gamma'$. By Euler's formula, $|E_0|-d+1=g$ and by Lemma \ref{l:1-diml-ker-at-Gamma'}, the fibers of $\Psi_{\Gamma,\eta}|_U$ are $1$-dimensional. So, $\Psi_{\Gamma,\eta}(U)$ has dimension 
\[
|E_0|-2g-1 = d+g-1-2g-1 = d-g-2.
\]
Combined with the fact that we have a $2$-dimensional space of choices for $\eta$, we see $\calR_{\Gamma}^{st}$ contains a $(d-g)$-dimensional set of divisors.
\end{proof}


\begin{remark}
The boundary of the region $\Omega_\eta$ plays a similar role to that of a \emph{string} that appeared in \cite[Lemma 4.2]{LR} and \cite[Proposition 4.49]{Markwig}.
\end{remark}

\subsection{Proof of Theorem \ref{thm:polyhedral} and the genus $0$ case of Theorem \ref{thm:main}}
\label{subsec:realizable-locus-polyhedral-and-g0}

Making use of Proposition \ref{prop:degreesOfFreedom}, we show that $\widetilde\calR_C$ is a polyhedral complex that is  balanced and of pure dimension $d-g$, thus proving Theorem \ref{thm:polyhedral}.
\begin{proof}[Proof of Theorem \ref{thm:polyhedral}]
Let $X$ be a toric compactification of $\mathbb{G}_m^2$ such that the closure $\overline{C}$ of $C$ is smooth. Consider the linear system on $X$ of curves whose Newton polygon is contained in the Newton polygon of $C$, and denote $\overline{L}$  its restriction to $\overline{C}$.  Let $S\subset\overline{L}$ be the subset of divisors that are supported on $C$, and let $L=S|_C$. By Lemma \ref{lem:anyCurve}, we see that $\trop(L)=\calR_C$. 

Corollary \ref{cor:anyCurve} and Proposition \ref{prop:degreesOfFreedom} imply that $\calR_C$ contains a $(d-g)$-dimensional subset. Since dimension is preserved under tropicalization, the dimension of $L$ is at least $d-g$ as well. On the other hand, by Riemann--Roch, its dimension is at most $d-g$. We conclude that the dimension is exactly $d-g$.

The pullback $\widetilde{L}$  of $L$ via the natural map $C^d\to C^d/{S^d}$  is embedded in a torus via $\widetilde{L}\subset C^d\subset \mathbb{G}_m^{2d}$.
Since tropicalization commutes with pullback, we have $\trop(\widetilde L)=\widetilde\calR_C$. Since $L$ has dimension $d-g$, every irreducible component $\widetilde{L}'$ of $\widetilde{L}$ is also $d-g$ dimensional, hence $\trop(\widetilde{L}')$ is a subset of $\widetilde\calR_C$ that is balanced and of pure dimension $d-g$. It follows that the union of the tropicalizations of these irreducible components, namely $\widetilde\calR_C$, is balanced and of pure dimension $d-g$ as well.
\end{proof}

We are now in a position to prove the genus 0 case of Theorem \ref{thm:main}.
\begin{corollary} 
\label{cor:g-0-liftable}
If $C$ is rational then every divisor of degree $d$ is realizable. That is, $\widetilde\calR_C = \mathrm{Eff}^d(\Gamma)$, the set of degree $d$ effective divisors on $\Gamma$.
\end{corollary}

\begin{proof}
Let ${L}$ be as in the proof of Theorem  \ref{thm:polyhedral}. Since $C$ is rational, we have  $\dim{L} = d$. But then $L$ is a full dimensional subset of $\text{Eff}^d(C)$, which is irreducible. It follows that $L = \text{Eff}^d(C)$. 
Since the map $\text{Eff}^d(C)\to\text{Eff}^d(\Gamma)$ is surjective, every divisor of degree $d$ on $\Gamma$ is realizable. 
\end{proof}

\section{Internal divisors}
\label{sec:genus1-case}
Having now handled Theorem \ref{thm:main} for rational curves, we turn to the genus $1$ case. Throughout this section, we assume that $\Gamma$ is a smooth tropical plane curve of genus $1$. We denote by $\mathcal{I}_\Gamma$ the collection of internal divisors (Definition \ref{def:internal}) in $|\self|$; 
Our strategy for proving Theorem \ref{thm:main} is as follows. In Lemma \ref{lem:neighbours}, we use the balancing condition for $\widetilde\calR_C$ (Theorem \ref{thm:polyhedral}) to show that if $\sigma_1$ and $\sigma_2$ are adjacent top-dimensional cells of $\mathcal{I}_\Gamma$, and if $\sigma_1\subset\calR_C$, then $\sigma_2\subset\calR_C$. Making use of Proposition \ref{prop:degreesOfFreedom}, we are able to find a top-dimensional cell of $\mathcal{I}_\Gamma$ contained in $\calR_C$, and then in Corollary \ref{cor:irreducible}, we use this to show $\mathcal{I}_\Gamma\subset\calR_C$, thereby proving Theorem \ref{thm:main} for $\Gamma$.

Let us now describe the polyhedral structure of $\internal$ and $|\self|$.
As $|\self|$ may be identified with a subset of $\RR^{2d}/S^{d}$, it has a natural polyhedral structure arising from the metric of $\RR^{2d}$.
Each cell is indexed by choices of vertices $v_1,\ldots,v_\ell$ and edges $e_{\ell+1},\ldots, e_d$. The corresponding cell parametrizes divisors in $|\self|$ where $\ell$ chips are forced to remain at the vertices $v_1,\ldots,v_\ell$, and the rest of the chips are allowed to vary in the edges $e_{\ell+1},\ldots,e_d$ (as long as they maintain linear equivalence with $\self$). A cell is said to be \emph{maximal} if it is not contained in a higher dimensional cell. Maximal cells are given by setting $\ell=0$ and letting all chips vary in the chosen edges. Note that $|\self|$ is not pure dimensional, i.e.~maximal cells may have different dimensions. For instance, a cell that parametrizes divisors supported only on bridge edges has dimension $d$, whereas any other cell has strictly smaller dimension.

Now, since $\internal$ excludes divisors supported away from the cycle, it is, in fact, pure dimensional of dimension $d-1$. Co-dimension $1$ cells of $\internal$ are obtained by forcing a single chip to remain at a vertex. A word of caution is in order. When there are exactly two chips on the cycle, linear equivalence implies that the position of one of them determines the position of the other. In particular, if the cycle is highly symmetric, forcing one chip to remain at a vertex could mean that the other chip is at a vertex throughout the entire cell. Such a situation requires special care, and we give it a name. 

\begin{definition}
A cell of $\internal$ is said to be \emph{exposed} if it parametrizes divisors of the form 
\[
\calD' = v' + v''+ p_3+\ldots + p_{d},
\]
where $v', v''$ are vertices of the cycle, and $p_3,p_4,\ldots, p_d$ are in the interiors of edges that are external to the cycle. 
\end{definition}

The term `exposed' is derived from the fact that there is an adjacent $d$-dimensional cell of $|\self|\setminus\internal$, where the two chips from the vertices are allowed to move away from the cycle. Exposed cells are  adjacent to two maximal cells in $\internal$, in which the 2 chips at the vertices move off in either direction.
For non-exposed co-dimension 1 cells, the adjacent maximal cells are given by letting the chip on the vertex move along one of the adjacent edges. Since $\Gamma$ is trivalent, there are three such cells in $|\self|$. 
The complexes $\internal$ and $|\self|$ coincide at a neighbourhood of every non-exposed cell.

Note that $\internal$ also excludes the cells of $|\self|$ where one chip is at a vertex, and the rest are away from the cycle. Such cells have dimension $d-1$, but are only adjacent to $d$ dimensional cells.  In particular, $\internal$ coincides with the union of maximal cells of $|\self|$ of dimension $d-1$. 

\begin{example}
The divisor depicted in Figure \ref{fig:exposed} is in an exposed $2$-dimensional cell, obtained by letting the two chips on the infinite rays vary. The adjacent $3$-dimensional cells of $\internal$ are given by also letting the chips on the vertices move towards each other or away from each other at equal speed. By letting those chips vary on the adjacent infinite rays, we obtain a $4$-dimensional cell of $|\self|$ which is not in $\internal$. 
\end{example}

\begin{figure}[h]
\centering
\begin{tikzpicture}[scale=.5]


\begin{scope}[shift={(0,0)}, scale=1]
\draw[blue] (-1.5,-1) -- (1.5,-1) -- (1.5,1) -- (-1.5,1) -- (-1.5,-1);
\draw[blue] (-1.5,-1) -- (-2.5,-2);
\draw[blue] (1.5,1) -- (2.5,2);
\draw[blue] (1.5,-1) -- (2.5,-2);
\draw[blue] (-1.5,1) -- (-2.5,2);

\draw [fill=purple,radius=.1] (-2,1.5) circle;
\draw [fill=purple,radius=.1] (1.7,1.2) circle;
\draw [fill=purple,radius=.1] (-1.5,-1) circle;
\draw [fill=purple,radius=.1] (1.5,-1) circle;

\end{scope}

\end{tikzpicture}
\caption{A divisor in an exposed cell of $|\self|$.}
\label{fig:exposed}
\label{ex:internal}
\end{figure}

As we shall now see, the set of internal divisors satisfies an additional desirable property.

\begin{lemma}\label{lem:con1}
If $\Gamma$ has genus $1$, then ${\internal}$ is connected in co-dimension 1.
\end{lemma}
\begin{proof}
Throughout the proof we say that a path in $\internal$ is \emph{admissible} if it avoids cells of co-dimension greater than $1$.


First, let $\calD_1$ and $\calD_2$ be divisors in top dimensional cells of $\internal$ that are supported on the cycle. In particular, they have the same number of chips on the cycle, and this number is at least three. If $\calD_2 = \calD_1 + \ddiv\phi$, then there is a path between them  in $|\self|$ given by $\calD(t) = \calD +  \ddiv\max({t,\phi})$. When $t$ is smaller than the minimum of $\phi$ we have $\calD(t) = \calD_1$, and when $t$ is greater than the max, we have $\calD(t) = \calD_2$. As we will see, this path can be perturbed to become admissible.

In any path in $|\self|$, the position of the chips vary continuously. Denote by $c_1(t), c_2(t), \ldots, c_d(t)$ the  functions describing the position of the different chips at  time $t$. Choose generic $\epsilon_1,\ldots,\epsilon_d\in\RR$ that are arbitrarily small and sum to zero. 
Let $\calD_1'$ and $\calD_2'$ be the divisors obtained from $\calD_1$ and $\calD_2$ by translating  each chip $c_i$ a distance $\epsilon_i$ along the cycle. Similarly, at time  $t$, consider the divisor $\calD(t)'$ obtained from $\calD(t)$ by translating each chip $c_i$ a distance $\epsilon_i$ along the cycle. All of these divisors are equivalent to each other, and $\calD(t)'$ is a continuous path between $\calD_1'$ and $\calD_2'$. Moreover, since the $\epsilon_i$'s were chosen generically, no more than one chip can be at a vertex at any time $t$. Therefore, the path obtained by concatenating the paths from $\calD_1$ to $\calD_1'$ to $\calD_2'$ to $\calD_2$ is admissible. 

The proof will be complete once we show that there is an admissible path from every divisor in $\internal$ to one supported on the cycle. Suppose that $\calD$ is in a top-dimensional cell of $\internal$, and has a chip at a point $p$ outside of the cycle.  Let $v$ be the point on the cycle that is closest to $p$. By continuously sliding the chip towards $v$ until it reaches the cycle, we obtain an admissible path from $\calD$ to $\calD-p+v$. Moreover, since $\calD$ was in $\internal$, the divisor $\calD-p+v$ has at least $2$ chips on the cycle. We may therefore move the chips on $v$ to the interior of an adjacent edge on the cycle in an admissible way. We have shown that there is an admissible path from $\calD$ to a divisor in a top cell with strictly fewer external chips. By induction, there is an admissible path to a divisor that is supported on the cycle. 
\end{proof}

As the next lemma shows, when a cell is realizable, its neighbouring cells are realizable as well.
 \begin{lemma}\label{lem:neighbours}
 Let $\sigma_1$ and $\sigma_2$ be top dimensional cells of $\internal$ intersecting at a co-dimension $1$ cell $\sigma_0$. 
 If $\sigma_1\subset\calR_C$, then $\sigma_2\subset\calR_C$.
  \end{lemma}
 \begin{proof}
In order to show that $\calR_C$ contains the adjacent cell $\sigma_2$ as well, we distinguish between the case where $\sigma_0$ is exposed and is adjacent to two cells of $\internal$, or $\sigma_0$ is not exposed, and is adjacent to three cells of $\internal$. We deal with the case of an exposed cell, which is harder, and leave the detail of non-exposed cells to the reader.\footnote{Here is why the non-exposed case is easier: when $\sigma_0$ is not exposed, it is contained in 3 maximal cells $\sigma_1$, $\sigma_2$, and $\sigma_3$. Since $\sigma_1\subset\calR_C$, the balancing condition forces $\sigma_2$ and $\sigma_3$ to be contained in $\calR_C$ as well. When $\sigma_0$ is exposed, a more refined argument is needed.}

Assume then that $\sigma_0$ is an exposed cell, which parametrizes divisors with two chips on vertices of the cycle, and the rest of the chips are in the interior of edges away from the cycle.  The two cells $\sigma_1$ and $\sigma_2$ correspond to divisors where the chips on the vertices move into the interior of cycle. We already know that $\calR_C$ contains $\sigma_1$. As we will see, $\calR_C$ may only be balanced if it contains $\sigma_2$ as well.

We set up some notation in order to describe the pullbacks of these cells to $\RR^{2d}$.
Suppose that $\sigma_0$ classifies divisors of the form $\calD' = v' + v''+ p_3+\ldots + p_{d}$, such that $v'$ and $v''$ are vertices, and $p_3,\ldots,p_d$ are in the interior of edges $e_3,\ldots, e_d$ that are all external to the cycle. For each cell $\sigma$, we denote $\tilde\sigma$ its pullback to $\RR^{2d}$.
Then we may write
\[
\tilde\sigma_0 = \{v'\}\times \{v''\} \times e_3\times\ldots\times e_d.
\]
 Let $f_1',f_2'\in\RR^2$ be the primitive direction vectors of the edges of the cycle emanating from $v'$, and $f_1'',f_2''$ the primitive direction vectors of the edges of the cycle emanating from $v''$. We choose $f_1'$ and $f''_2$ to have clockwise orientation along the cycle, and $f_1''$ and $f_2'$ to have counter-clockwise orientation along the cycle (see Figure \ref{fig:exposed}). If $\sigma_j$ (where $j=1,2$) is the cell in which the chips move from the vertices into $f'_j, f''_j$, then $\tilde\sigma_j$ is the subset of 
$\RR_{\geq 0}f_j'\times \RR_{\geq 0}f''_j\times e_3\times\ldots\times e_d$ in which the coefficients of the first two coordinates are equal.

 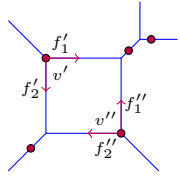
\begin{figure}[h]
\centering
\begin{tikzpicture}[scale=.5]

\begin{scope}[shift={(20,0)}]
\draw[blue] (-1,-1) -- (1,-1) -- (1,1) -- (-1,1) -- (-1,-1);
\draw[blue] (-1,-1) to (-2,-2);
\draw[blue] (1,1) to (1.5,1.5);
\draw[blue] (1.5,2.5) -- (1.5,1.5) -- (2.5,1.5);
\draw[blue] (-1,1) to (-2,2);
\draw[blue] (1,-1) to (2,-2);

\draw [fill=purple,radius=.1] (1.2,1.2) circle;
\draw [fill=purple,radius=.1] (1.8,1.5) circle;
\draw [fill=purple,radius=.1] (-1.4,-1.4) circle;

\draw [fill=purple,radius=.1] (-1,1) circle;
\node at (-.6,.6) {\tiny $v'$};
\draw[purple, ->] (-1,1) to (-1,.1);
\node  at (-1.4,0.2) {\tiny $f_2'$};
\draw[purple, ->] (-1,1) to (-.1,1);
\node  at (-0.6,1.4) {\tiny $f_1'$};

\draw [fill=purple,radius=.1] (1,-1) circle;
\node at (.6,-.6) {\tiny$v''$};
\draw[purple, ->] (1,-1) to (1,-.1);
\node  at (1.4,-0.2) {\tiny $f_1''$};
\draw[purple, ->] (1,-1) to (.1,-1);
\node  at (.6,-1.4) {\tiny $f_2''$};

\end{scope}
\end{tikzpicture}
\caption{A divisor in $\sigma_0$.}
\label{fig:exposed}
\end{figure}

We now describe the cell of $\widetilde{|\self|}\subset\RR^{2d}$ adjacent to $\sigma_0$ parametrizing divisors that are supported away from the cycle (note that this cell is $d$-dimensional). Since $\Gamma$ is balanced, the direction vector of the edge emanating from the cycle at $v_j$ is $-f_j'-f_j''$. This cell is therefore of the form
\[
\RR_{\leq 0}(f_1'+f_2')\times \RR_{\leq 0}(f''_1+f_2'') \times e_3\times\ldots\times e_d.
\]

Now, since $\widetilde\calR_C$ is of dimension $d-1$, its intersection with $\tilde\sigma_0$ consists of a finite number of $(d-1)$-dimensional cells $\tilde\tau_1,\tilde\tau_2,\ldots,\tilde\tau_k$.
Assume for the sake of contradiction that $\calR_C$ does not contain $\sigma_2$. We will show that in this case, the balancing condition cannot be satisfied. 
Fix an  integer vector $v_1$ such that $\tilde\sigma_1$ is spanned by $\tilde\sigma_0$ and $v_1$. Similarly, fix integer vectors $u_1,\ldots, u_k$ that span $\tilde\tau_i$. 
Then the balancing condition for $\widetilde\calR_C$ implies that 
\[
v_1+u_1+\ldots+u_k\in \tilde\sigma_0.
\]
Since each $u_i$ lives in $\tilde\tau_i$, its first two coordinates form a negative multiple of $f_1'+f_2'$, and its third and fourth coordinates   form a negative multiple of $f_1''+f_2''$. But $v_1$ lives in $\sigma_1$, so its first two coordinates form a positive multiple of $f_1'$, and the third and fourth coordinates form a positive multiple of   $f_1''$.  These vectors never sum to zero, so the balancing condition cannot be satisfied, which is a contradiction.
 \end{proof}
 
We are finally ready to prove that every internal divisor is realizable. Since we have assumed throughout this paper that the non-Archimedean valuation $\nu:K\to\RR$ is surjective, this proves Theorem \ref{thm:main} in the case where the valuation is surjective.

 \begin{corollary}\label{cor:irreducible}
 If $\Gamma$ has genus $1$, then 
\[
\internal\subseteq\calR_C.
\]
  \end{corollary}

\begin{proof}
To begin, we show that there is a top-dimensional cell of $\internal$ that is contained in $\calR_C$. Let $\eta\in\RR^2$ be generic with sufficiently small norm and let $\Gamma'=\Gamma+\eta$. Then the proof of Proposition \ref{prop:degreesOfFreedom} tells us $\calR_C$ contains a $(d-1)$-dimensional locus of divisors $\Gamma\cdot\Gamma''$ coming from sufficiently small perturbations of $\eta$ and of the edge lengths of $\Gamma$. Notice that all such $\Gamma\cdot\Gamma''$ are internal. We have therefore produced a $(d-1)$-dimensional open locus of $\calR_C$ contained in $\internal$, and hence there is a $(d-1)$-dimensional cell $\sigma_0$ of $\internal$ that contains an open set of $\calR_C$. We claim that $\calR_C$ must contain $\sigma_0$ itself. Indeed, if $\sigma_0\not\subset\calR_C$, then the balancing condition for $\widetilde\calR_C$ (Theorem \ref{thm:polyhedral}) dictates that there a cell $\sigma'$ of $\calR_C$ that is attached to the interior of $\sigma_0$. But $\calR_C$ is a subset of $|\self|$, and such a cell $\sigma'$ does not exist in $|\self|$, as $\sigma_0$ is a cell of $|\self|$.
%

Having now shown the existence of a top-dimensional cell $\sigma_0$ of $\internal$ with $\sigma_0\subseteq\calR_C$, we fix an interior point $p_0\in\sigma_0$. To prove $\internal\subseteq \calR_C$, it suffices to show that every top-dimensional cell $\sigma$ of $\internal$ is contained in $\calR_C$. Choose an interior point $p$ of $\sigma$. Since $\internal$ is connected in co-dimension 1 by Lemma \ref{lem:con1}, there is a path $\gamma\colon[0,1]\to\internal$ from $p_0$ to $p$ which avoids cells of co-dimension greater than one; choosing the path minimally, we can assume it intersects all cells a finite number of times. Any time that the path crosses from one cell to another, it must pass through their intersection. Since $\gamma$ avoids cells of co-dimension greater than one, this means we have $0=s_0<t_0\leq s_1<t_1\leq s_2<\dots<t_m=1$ and a chain of cells $\sigma_0\supset\tau_0\subset\sigma_1\supset\tau_1\subset\dots\subset\sigma_m=\sigma$ with the $\sigma_i$ full-dimensional, the $\tau_i$ of co-dimension 1, $\gamma([s_i,t_i])\subset\sigma_i$, and $\gamma([t_i,s_{i+1}])\subset\tau_i$; we can assume $\sigma_i\neq\sigma_{i+1}$.

Notice that if $\sigma_j\subseteq\calR_C$, then $\tau_j$ is also contained in $\calR_C$ since $\tau_j\subset\sigma_j$. Then Lemma \ref{lem:neighbours} tells us $\sigma_{i+1}\subseteq\calR_C$. Since $\sigma_0\subseteq\calR_C$, we see by induction that $\sigma\subseteq\calR_C$.
\end{proof}
 
To finish the proof of the main theorem, we remove the assumption that the valuation is surjective. 
Many of the arguments used above are invalid in this case, since the collection of 
$C$-realizable divisors  only becomes a polyhedral complex after passing to the closure. 

\begin{proof}[{Proof of Theorem \ref{thm:main}}]
Let $\calD$ be a $K$-rational internal divisor, and let $(\tilde K,\tilde\nu)$ be an extension of the non-Archimedean field $(K,\nu)$, in which the valuation $\tilde\nu$ surjects onto $\RR$. 

Let $\widetilde{N}$ and $N$ be the linear systems in $\GG_m^2(\tilde{K})$ and $\GG_m^2(K)$ respectively of curves whose Newton polygon is contained in that of $C$, and let $\widetilde{L}$ and $L$ be their restrictions to $C$. By Lemma \ref{lem:anyCurve}, we know that $\trop{\widetilde L}$ contains $\internal$, and in particular contains $\calD$. 

Let $D$ be a divisor on $C$ such that $\trop(D)=\calD$, and let $p_1,\ldots,p_d$ be the corresponding points  in $\GG_m^2({K})$. The collection of curves in $\widetilde{N}$ passing through these points is a vector subspace $\widetilde{V}$. Since $p_1,\ldots, p_d$ are $K$-points, $\widetilde{V}$ is defined via linear equations over $K$. The same linear equations cut out a vector subspace of $N$ of the same positive dimension. It follows that there is a curve $C'$ defined over $K$ passing through $p_1,\ldots,p_d$, and in particular, $C\cap C'$ is a divisor tropicalizing to $\calD$. 
\end{proof}

\section{Non-internal cells and generalizations}
 The techniques in the proof of Lemma \ref{lem:neighbours} may also be used to provide partial information about the realizable locus inside the non-internal cells $|\self|\setminus\internal$. In fact, the balancing condition of $\calR_C$ implies that there are exactly two possibilities for the realizable cells  adjacent to an exposed cell that are not in $\internal$.
 \begin{enumerate}
 \item Either there is a unique such cell, parametrizing divisors where the two chips on $v',v''$ move away from the cycle at equal speed, or
 \item There are two such cells, and in each of them one chip moves away from the cycle, and the other chip stays at the vertex.
 \end{enumerate}
 Note, however, that, since these cells are in a $d$-dimensional cell of $|\self|$, there is no guarantee that $\calR_C$ contains them  in their entirety.

\begin{example}
Let $\Gamma$ be the tropical curve shown in Figure \ref{fig:non_internal} with four vertices $v_1=(-1,1), v_2=(1,1), v_3=(1,-1), v_4=(-1,-1)$. Denote the infinite edge adjacent to $v_i$ by $e_i$ for $i=1,2,3,4$.  Let $C$ be any curve such that $\trop(C)=\Gamma$.  
Consider the set of divisors of the form $q_1+q_2 + p_3 + p_4$, where  $p_3, p_4$ are on $e_3,e_4$ respectively, and $q_1$ and $q_2$ are both in the cycle. By Theorem \ref{thm:main}, such a divisor is realizable if and only if $D$ is linearly equivalent to $\self$. 

Let $\sigma_0$ be cell parametrizing divisors obtained by forcing $q_1$ and $q_2$ to remain at $v_1,v_2$. In order for $\calR_C$ to be balanced at $\sigma_0$, there must be cells adjacent to it parametrizing divisors where the chips from $v_1,v_2$ move along $e_1,e_2$. Moreover, arbitrarily close to $\sigma_0$, precisely one of the following two options may occur.
 \begin{enumerate}
 \item Either there is a unique such cell, parametrizing divisors where the two chips on $v_1,v_2$ vertices move away from the cycle at equal speed, or
 \item There are two such cells, and in each of them one chip moves away from the cycle, and the other chip stays at the vertex.
 \end{enumerate}
We claim, moreover, that the only realizable option is the first one. Indeed, given a divisor close enough to $C_0$, we will realize it as the stable intersection with  a curve $\Gamma'$ with the same dual polygon as $\Gamma$. Lemma \ref{lem:anyCurve} then implies that this divisor is realizable. 

Indeed, let $D = p_1+p_2+p_3+p_4$ be a divisor such that $p_1$ and $p_2$ are at distance $\epsilon$ from $v_1,v_2$ on $e_1,e_2$, and $p_3, p_4$ are further at distances $\delta_3,\delta_4>\epsilon$ on $e_3,e_4$. 
Let $\Gamma'$ be the tropical curve whose vertices are at $(-1-\delta_4, 1+\epsilon), (1+\delta_3, 1+\epsilon), (-1-\delta_4, -R), (1+\delta_3, -R)$, where $R$ is any real number greater than $1+\max(\delta_3,\delta_4)$. Then $\Gamma\cap_{\text{st}}\Gamma' = D$. See Figure \ref{fig:non_internal}.

In particular, divisors where only one of $p_1$ or $p_2$ is away from the vertex are not realizable.  
\end{example}

Note that the result of the example did not depend on the choice of $C$. We have therefore  shown the following.

\begin{theorem}\label{thm:counter}
There is a tropical plane curve $\Gamma$ and a divisor $D\in|\self|$ such that $D$ is not the tropicalization of $C\cap C'$ for any pair of curves with $\trop(C)=\trop(C')=\Gamma$.

In particular, this is a counter example to \cite[Conjecture 3.4]{Mor15}, even in the case of self intersection (see \cite[Lemma 3.15]{BM} for a non self intersection example).
\end{theorem}

 \begin{figure}[h]
\centering
\begin{tikzpicture}[scale=.5]

\begin{scope}[shift={(0,0)}]

\node [red, left] at (3.3,.5) {\small $\Gamma'$};
\draw[red] (-2.4,1.2) -- (2.2,1.2) -- (2.2,-2.7) -- (-2.4,-2.7) -- (-2.4,1.2);
\draw[red] (-2.4,1.2) to (-2.8,1.6);
\draw[red] (2.2,1.2) to (2.6,1.6);
\draw[red] (2.2,-2.7) to (2.6,-3.1);
\draw[red] (-2.4,-2.7) to (-2.8,-3.1);

\node [blue, left] at (2.2,2.4) {\small $\Gamma$};
\draw[blue] (-1,-1) -- (1,-1) -- (1,1) -- (-1,1) -- (-1,-1);
\draw[blue] (-1,-1) to (-3,-3);
\draw[blue] (1,1) to (2.5,2.5);
\draw[blue] (-1,1) to (-2.5,2.5);
\draw[blue] (1,-1) to (3,-3);

\draw [fill=purple,radius=.1] (1.2,1.2) circle;
\node [above] at (1.2,1.2) {\small $p_2$};

\draw [fill=purple,radius=.1] (-1.2,1.2) circle;
\node [above] at (-1.2,1.2) {\small $p_1$};

\draw [fill=purple,radius=.1] (-2.4,-2.4) circle;
\node [left] at (-2.4,-2.4) {\small $p_4$};

\draw [fill=purple,radius=.1] (2.2,-2.2) circle;
\node [right] at (2.2,-2.2) {\small $p_3$};

\end{scope}
\end{tikzpicture}
\caption{Realizing a non-internal divisor via intersection.}
\label{fig:non_internal}
\end{figure}
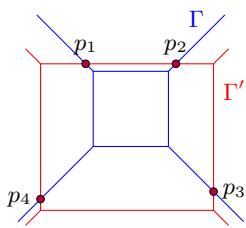

We end the paper with a natural generalization of our results to higher genus.
Let $\Gamma$ be a tropical curve of genus $g$, and let $\sigma$ be a cell of  $|\self|$ parametrizing divisors with $k$ chips on vertices. If $\dim\sigma = d-g-k$ then the divisors in $\sigma$ are said to be \emph{internal}. 
In genus $1$, this definition coincides with our original definition. 

\begin{conjecture}
Let $\Gamma$ be a smooth tropicalization of an algebraic curve $C$. Then every internal divisor on $\Gamma$ is $C$-realizable.
\end{conjecture}

	\bibliographystyle{alpha}
	\bibliography{bibfile}	

\begin{thebibliography}{MUW17}

\bibitem[BGS11]{recession-cones-form-fan}
J.~Burgos~Gil and M.~Sombra.
\newblock When do the recession cones of a polyhedral complex form a fan?
\newblock {\em Discrete Comput. Geom.}, 46(4):789--798, 2011.

\bibitem[BIMS15]{BIMS}
E.~Brugall\'{e}, I.~Itenberg, G.~Mikhalkin, and K.~Shaw.
\newblock Brief introduction to tropical geometry.
\newblock In {\em Proceedings of the {G}\"{o}kova {G}eometry-{T}opology
  {C}onference 2014}, pages 1--75. G\"{o}kova Geometry/Topology Conference
  (GGT), G\"{o}kova, 2015.

\bibitem[BJ16]{BJ}
M.~Baker and D.~Jensen.
\newblock {\em Degeneration of Linear Series from the Tropical Point of View
  and Applications}, pages 365--433.
\newblock Springer International Publishing, Cham, 2016.

\bibitem[BM11]{BM}
E.~Brugall{\'e} and L.~Medrano.
\newblock Inflection points of real and tropical plane curves.
\newblock {\em Journal of Singularities}, 4:74--103, Feb 2011.

\bibitem[CJP15]{CJP}
D.~Cartwright, D.~Jensen, and S.~Payne.
\newblock Lifting divisors on a generic chain of loops.
\newblock {\em Canad. Math. Bull.}, 58(2):250--262, 2015.

\bibitem[CLM15]{CLM}
L.~Caporaso, Y.~Len, and M.~Melo.
\newblock Algebraic and combinatorial rank of divisors on finite graphs.
\newblock {\em J. Math. Pures Appl. (9)}, 104(2):227--257, 2015.

\bibitem[CM16]{CM14}
M.A. Cueto and H.~Markwig.
\newblock How to repair tropicalizations of plane curves using modifications.
\newblock {\em Exp. Math.}, 25(2):130--164, 2016.

\bibitem[CMR16]{CMR}
R.~Cavalieri, H.~Markwig, and D.~Ranganathan.
\newblock {Tropicalizing the space of admissible covers.}
\newblock {\em {Math. Ann.}}, 364(3-4):1275--1313, 2016.

\bibitem[He16]{He}
X.~He.
\newblock A generalization of lifting non-proper tropical intersections.
\newblock {\em Preprint arXiv:1606.04455}, 06 2016.

\bibitem[HMY12]{HMY}
C.~Haase, G.~Musiker, and J.~Yu.
\newblock Linear systems on tropical curves.
\newblock {\em Math. Z.}, 270(3-4):1111--1140, 2012.

\bibitem[IL18]{IL}
N.~Ilten and Y.~Len.
\newblock Projective duals to algebraic and tropical hypersurfaces.
\newblock {\em Preprint arXiv:1803.08912}, 03 2018.

\bibitem[LM19]{LM}
Y.~Len and H.~Markwig.
\newblock Lifting tropical bitangents.
\newblock {\em Journal of Symbolic Computation}, 03 2019.

\bibitem[LR18]{LR}
Y.~Len and D.~Ranganathan.
\newblock Enumerative geometry of elliptic curves on toric surfaces.
\newblock {\em Israel Journal of Mathematics}, 226(1):351--385, Jun 2018.

\bibitem[Mar06]{Markwig}
H.~Markwig.
\newblock {\em The enumeration of plane tropical curves}.
\newblock PhD thesis, Technische Universit\"at Kaiserslautern, 2006.

\bibitem[Mik06]{Mi06}
G.~Mikhalkin.
\newblock Tropical geometry and its applications.
\newblock In M.~Sanz-Sole et~al., editor, {\em Invited lectures v.\ II,
  Proceedings of the ICM Madrid}, pages 827--852, 2006.
\newblock arXiv:math.AG/0601041.

\bibitem[Mor15]{Mor15}
R.~Morrison.
\newblock Tropical images of intersection points.
\newblock {\em Collect. Math.}, 66(2):273--283, 2015.

\bibitem[MS15]{MS}
D.~Maclagan and B.~Sturmfels.
\newblock {\em {Introduction to tropical geometry.}}
\newblock Providence, RI: American Mathematical Society (AMS), 2015.

\bibitem[MUW17]{MUW}
M.~Moeller, M.~Ulirsch, and A.~Werner.
\newblock Realizability of tropical canonical divisors.
\newblock {\em preprint arXiv: 1710.06401}, Oct 2017.

\bibitem[NPS16]{NPS}
J.~Nicaise, S.~Payne, and F.~Schroeter.
\newblock Tropical refined curve counting via motivic integration.
\newblock {\em Preprint arXiv:1603.08424}, March 2016.

\bibitem[NS16]{SN}
M.~Nisse and F.~Sottile.
\newblock Higher convexity for complements of tropical varieties.
\newblock {\em Mathematische Annalen}, 365(1):1--12, Jun 2016.

\bibitem[OP13]{OP}
B.~Osserman and S.~Payne.
\newblock Lifting tropical intersections.
\newblock {\em Doc. Math.}, 18:121--175, 2013.

\bibitem[OR13]{OR}
B.~Osserman and J.~Rabinoff.
\newblock Lifting nonproper tropical intersections.
\newblock In {\em Tropical and non-{A}rchimedean geometry}, volume 605 of {\em
  Contemp. Math.}, pages 15--44. Amer. Math. Soc., Providence, RI, 2013.

\bibitem[Rau09]{Rau}
J.~Rau.
\newblock Intersections on tropical moduli spaces.
\newblock {\em Rocky Mountain Journal of Mathematics}, 46, 01 2009.

\end{thebibliography}

\end{document}